\documentclass{amsart}
\usepackage{amsfonts}
\usepackage{pifont}
\usepackage{amsmath}
\usepackage{graphicx,hyperref,fancyref}
\usepackage{cleveref}

\newtheorem{theorem}{Theorem}
\theoremstyle{plain}

\newtheorem{corollary}{Corollary}

\newtheorem{definition}{Definition}

\newtheorem{lemma}{Lemma}

\newtheorem{proposition}{Proposition}
\newtheorem{remark}{Remark}

\numberwithin{equation}{section}

\begin{document}
\title[$p $-variational Calculus]{$p $-variational Calculus}
\author{\.{I}lker Gen\c{c}t\"{u}rk}
\address{Department of Mathematics, K{\i }r{\i }kkale University,
	Yahsihan,
	71450 K{\i }r{\i }kkale, Turkey}
\email{ilkergencturk@gmail.com}

\date{\today }

\begin{abstract}
The aim of this paper is to bring together a new type of quantum calculus, namely $p $-calculus, and variational calculus. We develop $p $-variational calculus  and obtain a necessary
optimality condition of Euler-Lagrange type and a sufficient optimality
condition.
\end{abstract}

\maketitle

%\dedicatory{Dedicated to the memory of S. Bach.}
%\thanks{This paper is in final form and no version of it will be submitted for publication elsewhere.}

\section{Introduction}
One of interesting fields in mathematics is quantum calculus known as calculus without limits. It is well-known that it changes the classical derivative by a quantum difference operator. It has important applications in physics and chemical physics. Moreover, it plays an important role in several fields of mathematics such as orthogonal polynomials, analytic number theory, geometric function theory, combinatorics etc. \cite{AnnMan,HarmanI,HarmanII,KocaGenAyd,pashaev2014q}.

There are several types of quantum calculus such as $ h $-calculus, $ q $-calculus and the others which some generalizations of $ h $-calculus and $ q $-calculus. In the beginning of the twentieth century, Jackson introduced the $ q $-calculus with following notation $$ \frac{f(qt)-f(t)}{(q-1)t} $$ where $ q $ is a fixed number different from $ 1 $, $ t \neg 0 $ and $ f $ is areal function. It is clear that if $ f $ is differentiable at $ t \neq 0 $, then $f'(t)=\lim\limits_{q \to 1} \frac{f(qt)-f(t)}{(q-1)t}$. We refer to reader to \cite{KacCheung}, \cite{Ernstcomp}, for basic concepts of quantum calculus and history of $ q $-calculus.

In \cite{NeaMeh}, authors produced a new type of quantum calculus with the following expression $$ \frac{f(t^p)-f(t)}{t^p-t} $$ and above notation is definition of the $ p $-derivative. Moreover, some new properties of functions and Steffensen inequality in $ p $-calculus \cite{NeaMeh2,yadollahzadeh2019steffensen} and $ pq $-calculus which a generalization of $ p $-calculus \cite{genccturk2019new} were given.

The calculus of variation is one of the classical subjects in mathematics and it establishes the relation between other branches of mathematics such as differential equations, geometry, and physics. Also it has important applications in mechanics, engineering, economics, biology and electrical engineering. The calculus of variation deals with finding extrema and, in this sense, one can say that it is a branch of optimization. Because of its importance, studies based in quantum calculus are occurred. In this regard, we refer \cite{aldwoah2011power,bangerezako2004variational,da2012q,da2018general,da2012higher} to readers to for details in calculus of variations based in different quantum operators.

The main objective of this paper is to provide a necessary optimality
condition and a sufficient optimality condition for the $p$-variational
problem
\begin{align}
\mathcal{L}[y]
&=\int \limits_{a}^{b}L\bigg(t,y(t^{p}),D_{p}[y](t)
\bigg)d_{p}t\longrightarrow extremize  \notag \\
&  \tag*{(P)}  \label{problem} \\
y &\in \mathcal{Y}^{{}}\big( [a, b ]_{p},%
%TCIMACRO{\U{211d} }%
%BeginExpansion
\mathbb{R}
%EndExpansion
\big),~y(a)=\alpha,~y(b)=\beta   \notag
\end{align}%
where a and b are fixed real numbers and by extremize, we mean minimize or
maximize. Problem \ref{problem} with detailed  will be given in  Section 3. %
Additionally, Lagrangian $L$ has the desired the following hypotheses:

\begin{enumerate}
	\item[(H1)\label{H1}] $(u,v)\rightarrow L(t,u,v)$ is a $C^{1}(%
	%TCIMACRO{\U{211d} }%
	%BeginExpansion
	\mathbb{R}
	%EndExpansion
	^{2},%
	%TCIMACRO{\U{211d} }%
	%BeginExpansion
	\mathbb{R}
	%EndExpansion
	)$ function for any $t\in I;$
	
	\item[(H2)]\label{H2} $t\rightarrow L\big(t,y(t^{p}),D_{p}[y](t)\big)$ is continuous at$1$ for
	any admissible function $y;$
	
	\item[(H3)] \label{H3} functions $t\rightarrow \partial _{i+2}L\big(t,y(t^{p}),D_{p}[y](t)\big)$
	belongs to $\mathcal{Y}^{{}}\big( [a,b]_{p},%
	%TCIMACRO{\U{211d} }%
	%BeginExpansion
	\mathbb{R}
	%EndExpansion
	\big) $ for all admissible $y$, $i=0,1;$
\end{enumerate}
where $I$ is an interval of $%
%TCIMACRO{\U{211d} }%
%BeginExpansion
\mathbb{R}
%EndExpansion
$ containing 1; $a,b\in I,$ $a<b,$ and $\partial _{j}L$ denotes the partial
derivative of $L$ with respect to its $j$th argument.

This paper is organized as follows. In section 2 we call up some necessary definitions and  theorems about $ p $-calculus. In section 3, we give our results for the $ p $-variational calculus.

\section{Preliminaries}

Let $p\in (0,1)$ be and consider interval $J$ =$[0,\infty )$ . We will
denote by $J_{p}$ the set $J_{p}:=\{x^{p}:x\in J\}$. Throughout this paper,
we assume that function $f(x)$ is defined on $J$.

We need some definitions and fundamental results on $p$-calculus to prove
our results.\cite{NeaMeh}

\begin{definition}
	Consider an arbitrary function $f(x)$. Its $p$-derivative is defined as%
	\begin{equation*}
	D_{p}f(x)=\frac{f(x^{p})-f(x)}{x^{p}-x},\text{ if }x\neq 0,1,
	\end{equation*}%
	and%
	\begin{equation*}
	D_{p}f(0)=\lim_{x\rightarrow 0^{+}}D_{p}f(x),\text{ \ }D_{p}f(1)=\lim_{x%
		\rightarrow 1}D_{p}f(x).
	\end{equation*}
\end{definition}

\begin{corollary}
	If $f(x)$ is differentiable, then~$\lim \limits_{p\rightarrow
		1}D_{p}f(x)=f^{\prime }(x),$ and also if $f^{\prime }(x)$ exists in a
	neighborhood of $x=0$, $x=1$ and is continuous at $x=0$ and $x=1$, then we
	have%
	\begin{equation*}
	D_{p}f(0)=f_{+}^{\prime }(0),~D_{p}f(1)=f^{\prime }(1).
	\end{equation*}
\end{corollary}

\begin{definition}
	The $p$-derivative of higher order of function $f(x)$ is defined by%
	\begin{equation*}
	\left( D_{p}^{0}f\right) (x)=f(x),~\left( D_{p}^{n}f\right)
	(x)=D_{p}\left( D_{p}^{n-1}f\right) (x),\text{ }n\in
	%TCIMACRO{\U{2115} }%
	%BeginExpansion
	\mathbb{N}
	%EndExpansion
	.
	\end{equation*}
\end{definition}

We note that the $p$-derivative has the following properties.

\begin{theorem}
	Let $f$ and $g$ be $p$-differentiable on $J$, let $\alpha $ and $\beta \in
	%TCIMACRO{\U{211d} }%
	%BeginExpansion
	\mathbb{R}
	%EndExpansion
	$ and $x\in J_{p}$. Then
	
	\begin{enumerate}
		\item $D_{p}f\equiv 0,$ then $f$ is a constant. Conversely, $D_{p}c\equiv 0$
		for any $c.$
		
		\item $D_{p}(\alpha f+\beta g)(x)=\alpha D_{p}f(x)+\beta D_{p}g(x).$
		
		\item $D_{p}(f(x)g(x))=g(x^{p})D_{p}f(x)+f(x)D_{p}g(x).$
		
		\item $$D_{p}\left( \frac{f}{g}\right) (x)=\frac{g(x)D_{p}f(x)-f(x)D_{p}g(x)}{%
			g(x)g(x^{p})}.$$
	\end{enumerate}
\end{theorem}

\begin{definition}
	A function $F(x)$ is a $p$-antiderivative of $f(x)$ if $D_{p}F(x)=f(x)$. It
	is denoted by%
	\begin{equation*}
	F(x)=\int f(x)d_{p}x.
	\end{equation*}
\end{definition}

\begin{definition}
	The $p$-integral of $f(x)$ is defined to be the series%
	\begin{equation*}
	\sum \limits_{j=0}^{\infty }(x^{p^{j}}-x^{p^{j+1}})f(x^{p^{j}}).
	\end{equation*}
\end{definition}

In \cite{NeaMeh}, authors considered the following three cases to define the
definite $p $-integral.

\begin{enumerate}
	\item[\textbf{Case 1.}] Let $1<a<b$ where $a,b$ $\in
	%TCIMACRO{\U{211d} }%
	%BeginExpansion
	\mathbb{R}
	%EndExpansion
	$ and function $f$ is defined on $(1,b].$ Notice that for any $j\in
	\{0,1,2,3,\ldots \},$ $b^{p^{j}}\in (1,b].$
	
	\begin{definition}
		The $p$-integral of a function $f(x)$ on the interval $(1,b]$ is defined as%
		\begin{equation}
		\int \limits_{1}^{b}f(x)d_{p}x=\lim_{n\rightarrow \infty }\sum
		\limits_{j=0}^{N}(b^{p^{j}}-b^{p^{j+1}})f(b^{p^{j}})=\sum
		\limits_{j=0}^{\infty }(b^{p^{j}}-b^{p^{j+1}})f(b^{p^{j}}),
		\label{case1defint}
		\end{equation}%
		and%
		\begin{equation*}
		\int \limits_{a}^{b}f(x)d_{p}x:=\int \limits_{1}^{b}f(x)d_{p}x-\int
		\limits_{1}^{a}f(x)d_{p}x.
		\end{equation*}
	\end{definition}
	
	\item[\textbf{Case 2.}] Let $0<b<1$ where $b$ $\in
	%TCIMACRO{\U{211d} }%
	%BeginExpansion
	\mathbb{R}
	%EndExpansion
	$ $.$ Notice that for any $j\in \{0,1,2,3,\ldots \},$ $b^{p^{j}}\in \lbrack
	b,1)$ and $b^{p^{j}}<b^{p^{j+1}}.$
	
	\begin{definition}
		The $p$-integral of a function $f(x)$ on the interval $[b,1)$ is defined as%
		\begin{equation*}
		\int \limits_{b}^{1}f(x)d_{p}x=\lim_{n\rightarrow \infty }\sum
		\limits_{j=0}^{N}(b^{p^{j+1}}-b^{p^{j}})f(b^{p^{j}})=\sum
		\limits_{j=0}^{\infty }(b^{p^{j+1}}-b^{p^{j}})f(b^{p^{j}}).
		\end{equation*}%
		The $p$-integrals defined above are also denoted by
		\begin{equation*}
		\int \limits_{1}^{b}f(x)d_{p}x=I_{p^{+}}f(b),\text{ \  \ }\int
		\limits_{b}^{1}f(x)d_{p}x=I_{p^{-}}f(b).
		\end{equation*}
	\end{definition}
	
	\item[\textbf{Case 3}] Let $0<a<b<1$ where $a,b$ $\in
	%TCIMACRO{\U{211d} }%
	%BeginExpansion
	\mathbb{R}
	%EndExpansion
	$ $.$ Notice that for any $j\in \{0,1,2,3,\ldots \},$ $b^{p^{-j}}\in (0,b]$
	and $b^{p^{-j-1}}<b^{p^{-j}}.$
	
	\begin{definition}
		The $p$-integral of a function $f(x)$ on the interval $(0,b](b<1)$ is
		defined as%
		\begin{eqnarray*}
			I_{p}f(b) &=&\int \limits_{0}^{b}f(x)d_{p}x=\lim_{n\rightarrow \infty }\sum
			\limits_{j=0}^{N}(b^{p^{-j}}-b^{p^{-j-1}})f(b^{p^{-j-1}}) \\
			&=&\sum \limits_{j=0}^{\infty }(b^{p^{-j}}-b^{p^{-j-1}})f(b^{p^{-j-1}})
		\end{eqnarray*}%
		and%
		\begin{equation*}
		\int \limits_{a}^{b}f(x)d_{p}x:=\int \limits_{0}^{b}f(x)d_{p}x-\int
		\limits_{0}^{a}f(x)d_{p}x.
		\end{equation*}
	\end{definition}
\end{enumerate}

\begin{remark}\cite{NeaMeh}
	If $p\in (0,1),$ then for any $j\in \{0,\pm 1,\pm 2,\pm 3,\ldots \},$ we
	have $p^{p^{j}}\in (0,1),$ $p^{p^{j}}<p^{p^{j+1}}$ and
	\begin{equation}\label{01int}
	\int \limits_{0}^{1}f(x)d_{p}x=\sum \limits_{j=-\infty }^{\infty }\int
	\limits_{p^{p^{j}}}^{p^{p^{j+1}}}f(x)d_{p}x=\sum \limits_{j=-\infty
	}^{\infty }(p^{p^{j+1}}-p^{p^{j}})f(p^{p^{j}}).
	\end{equation}
\end{remark}

By definitions of $p $-integrals in \cite{NeaMeh}, we give a more general formula:

\begin{corollary}[Cf. Corollary 4.12, \cite{NeaMeh}]
	Suppose $0 \leq a<1<b $. The $p$-integral of a function $f(x)$ on the
	interval $[a,b]$ is defined as
	\begin{equation}  \label{key}
	\int \limits_{a}^{b}f(t)d_{p}t=\int \limits_{0}^{b}f(t)d_{p}t -\int
	\limits_{0}^{a}f(t)d_{p}t
	\end{equation}
	where
	\begin{equation}
	\int \limits_{0}^{x}f(t)d_{p}t=\sum \limits_{j=0}^{\infty
	}(x^{p^{-j}}-x^{p^{-j-1}})f(x^{p^{-j-1}}), ~\text{if } 0 \leq x <1,
	\end{equation}
	and
	\begin{align}
	\int \limits_{0}^{x}f(t)d_{p}t=&\int \limits_{0}^{1}f(t)d_{p}t+\int \limits_{1}^{x}f(t)d_{p}t\\
	=&\sum \limits_{j=-\infty
	}^{\infty }(p^{p^{j+1}}-p^{p^{j}})f(p^{p^{j}}) \\
	&+\sum
	\limits_{j=0}^{\infty }(x^{p^{j+1}}-x^{p^{j}})f(x^{p^{j}}) ~\text{if }  x >1,
	\end{align}
	provided the series converges at $x=a $ and $x=b $. In that case, $f $ is
	said to be $p $-integrable on $[a,b] $.
\end{corollary}

\begin{definition}
	The $p$-integral of higher order of a function $f$ is given by%
	\begin{equation*}
	\left( I_{p}^{0}f\right) (x)=f(x),\text{ \ }\left( I_{p}^{n}f\right)
	(x)=I_{p}\left( I_{p}^{n-1}f\right) (x),\text{ }n\in
	%TCIMACRO{\U{2115} }%
	%BeginExpansion
	\mathbb{N}
	%EndExpansion
	.
	\end{equation*}

\end{definition}
Following lemmas are given to obtain fundamental theorem of $ p $-calculus by authors in \cite{NeaMeh}.

\begin{lemma}
	If $x>1$ and $p\in (0,1)$, then $D_{p}I_{p^{+}}f(x)=f(x)$, and also if
	function $f$ is continuous at $x=1$, then we have $%
	I_{p^{+}}D_{p}f(x)=f(x)-f(1).$
\end{lemma}

\begin{lemma}
	If $x,p\in (0,1),$ then $D_{p}I_{p^{-}}f(x)=-f(x)$, and also if function $f$
	is continuous at $x=1$, then we have $I_{p^{-}}D_{p}f(x)=f(1)-f(x).$
\end{lemma}

\begin{lemma}
	If $x,p\in (0,1)$ and $I_{p}f(x)=\int \limits_{0}^{x}f(s)d_{p}s,$ then $%
	D_{p}I_{p}f(x)=f(x)$, and also if function $f$ is continuous at $x=0$, then
	we have $I_{p}D_{p}f(x)=f(x)-f(0).$
\end{lemma}

\begin{theorem}[\textbf{Fundamental theorem of }$p$\textbf{-calculus},\cite{NeaMeh}]
	Let $p\in (0,1)$. If $F(x)$ is an antiderivative of $f(x)$ and $F(x)$ is
	continuous at $x=0$ and $x=1$, then for every $0\leq a<b\leq \infty $, we
	have
	\begin{equation*}
	\int \limits_{a}^{b}f(x)d_{p}x=F(b)-F(a).
	\end{equation*}
\end{theorem}

\begin{corollary}
	If $f(x)$ is continuous at $x=0$ and $x=1$, then we have
	\begin{equation*}
	\int \limits_{a}^{b}D_{p}f(x)d_{p}x=f(b)-f(a).
	\end{equation*}
\end{corollary}

\begin{corollary}
	If $f(x)$ and $g(x)$ is continuous at $x=0$ and $x=1$, then we have%
	\begin{equation} \label{byparts}
	\int
	\limits_{a}^{b}f(x)D_{p}g(x)d_{p}x=f(b)g(b)-f(a)g(a)-\int%
	\limits_{a}^{b}g(x^{p})D_{p}f(x)d_{p}x.
	\end{equation}%
	This formula is called $p$-integration by parts.
\end{corollary}

The  $p$-integral has the following properties:

\begin{theorem}
	Let $f,g:I\rightarrow
	%TCIMACRO{\U{211d} }%
	%BeginExpansion
	\mathbb{R}
	%EndExpansion
	$ be $p$-integrable on $J\,,a,b,c\in J$ and $\  \alpha ,\beta \in
	%TCIMACRO{\U{211d} }%
	%BeginExpansion
	\mathbb{R}
	%EndExpansion
	.$Then
	
	\begin{enumerate}
		\item $\int \limits_{a}^{a}f(t)d_{p}t=0;$
		
		\item $\int \limits_{a}^{b}f(t)d_{p}t=-\int \limits_{b}^{a}f(t)d_{p}t;$
		
		\item $\int \limits_{a}^{b}f(t)d_{p}t=\int \limits_{a}^{c}f(t)d_{p}t+\int
		\limits_{c}^{b}f(t)d_{p}t;$
		
		\item $\int \limits_{a}^{b}\left( \alpha f+\beta g\right) (t)d_{p}t=\alpha
		\int \limits_{a}^{b}f(t)d_{p}t+\beta \int \limits_{a}^{b}g(t)d_{p}t.$
	\end{enumerate}
\end{theorem}

\begin{proof}
	By using definitions of definite $ p $-integral, the proof is clear.
\end{proof}

In what follows, for a given $t\in J$, we denote
\begin{equation*}
\lbrack t]_{p}:=\{t^{p^{j}}:j\in N_{0}\} \cup \{0,1\},
\end{equation*}%
and
\begin{equation*}
\lbrack a,b]_{p}:=[a]_{p}\cup \lbrack b]_{p}.
\end{equation*}

Because of different definitions of $ p $-integral depending on interval, from now on, we assume that  $0<a<1<b$ for $a,b\in J$.

\begin{proposition}
	Let $f$ and $g$ be $p$-integrable on $J$ and $a,b\in J$ such
	that $0<a<1<b$. If $|f(t)|\leq g(t)$ for all $t\in \lbrack a,b]_{p}$, then
	for $x,y\in \lbrack a,b]_{p}$, $x<1<y,$ we have
	
	\begin{enumerate}
		\item
		\begin{equation*}
		\left \vert \int \limits_{1}^{y}f(t)d_{p}t\right \vert \leq \int
		\limits_{1}^{y}g(t)d_{p}t,
		\end{equation*}
		
		\item
		\begin{equation*}
		\left \vert \int \limits_{1}^{x}f(t)d_{p}t\right \vert \leq -\int
		\limits_{1}^{x}g(t)d_{p}t.
		\end{equation*}
		
		\item
		\begin{equation*}
		\left \vert \int \limits_{x}^{y}f(t)d_{p}t\right \vert \leq \int
		\limits_{x}^{y}g(t)d_{p}t.
		\end{equation*}
		
		Consequently, if $g(t)\geq 0$ for all $t\in \lbrack a,b]_{p}$, then the
		inequalities $$\int \limits_{1}^{b}g(t)d_{p}t\geq 0 \text{ and} \int%
		\limits_{a}^{b}g(t)d_{p}t\geq 0$$ hold for all for $x,y\in \lbrack
		a,b]_{p}$.
	\end{enumerate}
\end{proposition}

\begin{proof}
	\begin{enumerate}
		\item Since $y>1 $, then $y^{p^{j+1}}<y^{p^{j}},$ $j\in N_{0}$, $y\in
		\lbrack a,b]_{p},$ \
		\begin{eqnarray*}
			\left \vert \int \limits_{1}^{y}f(t)d_{p}t\right \vert &=&\left \vert \sum
			\limits_{j=0}^{\infty }(y^{p^{j}}-y^{p^{j+1}})f(y^{p^{j}})\right \vert \leq
			\sum \limits_{j=0}^{\infty }(y^{p^{j}}-y^{p^{j+1}})\left \vert
			f(y^{p^{j}})\right \vert \\
			&\leq &\sum \limits_{j=0}^{\infty }(y^{p^{j}}-y^{p^{j+1}})g(y^{p^{j}})=\int
			\limits_{1}^{y}g(t)d_{p}t.
		\end{eqnarray*}
		
		\item
		\begin{eqnarray*}
			\left \vert \int \limits_{x}^{1}f(t)d_{p}t\right \vert &=&\left \vert \sum
			\limits_{j=0}^{\infty }(x^{p^{j+1}}-x^{p^{j}})f(x^{p^{j}})\right \vert \leq
			\sum \limits_{j=0}^{\infty }(x^{p^{j}}-x^{p^{j+1}})\left \vert
			f(x^{p^{j}})\right \vert \\
			&\leq &-\sum \limits_{j=0}^{\infty
			}(x^{p^{j}}-x^{p^{j+1}})g(x^{p^{j}})=-\int \limits_{1}^{x}g(t)d_{p}t.
		\end{eqnarray*}
		
		\item The proof is similar previous ones.
	\end{enumerate}
\end{proof}

\section{Main Results}

The main objective of this section is to introduce the $p$-variational
calculus. For this purpose, we consider the following variational
problem:%

\begin{align}
\mathcal{L}[y]
&=\int \limits_{a}^{b}L\big(t,y(t^{p}),D_{p}[y](t)
\big)d_{p}t\longrightarrow extremize  \notag \\
&  \tag*{(P)}  \label{problem} \\
y &\in \mathcal{Y}^{{}}\big( [a, b ]_{p},%
%TCIMACRO{\U{211d} }%
%BeginExpansion
\mathbb{R}
%EndExpansion
\big),~y(a)=\alpha,~y(b)=\beta   \notag
\end{align}
where by extremize we mean minimize or maximize and $y\in \mathcal{Y}$, where%
\begin{equation*}
\mathcal{Y}:=\{y:I\rightarrow
%TCIMACRO{\U{211d} }%
%BeginExpansion
\mathbb{R}
%EndExpansion
\text{ }|\text{ }y\text{ and }D_{p}[y]\text{ are bounded on }[a , b
]_{p}\text{ and continuous at }0 \text{~and~} 1\}
\end{equation*}%
equipped with the norm
\begin{equation*}
||y||=\sup_{t\in \lbrack a , b ]_{p}}|y(t)|+\sup_{t\in \lbrack
	a , b ]_{p}}|D_{p}[y](t)|.
\end{equation*}

\begin{definition}
	We say that $y$ is an admissible function for problem \ref{problem} if $y\in \mathcal{%
		Y }^{}\left( [a , b]_{p},
	%TCIMACRO{\U{211d} }%
	%BeginExpansion
	\mathbb{R}
	%EndExpansion
	\right) $ and $y$ satisfies the boundary conditions $y(a)=\alpha,~y(b)=\beta$.
\end{definition}

\begin{definition}
	We say that $y_{\ast }$ is a local minimizer (resp. local maximizer) for
	problem \ref{problem} if $y_{\ast }$ is an admissible function and there exists $%
	\delta >0$ such that
	\begin{equation*}
	\mathcal{L} \lbrack y_{\ast }]\leq \mathcal{L} \lbrack y]\text{ \  \ }(\text{resp. }\mathcal{L} \lbrack y_{\ast }]\geq \mathcal{L} \lbrack y])
	\end{equation*}
	for all admissible $y$ with $||y_{\ast }-y||_{}<\delta $.
\end{definition}

\begin{definition}
	We say that $\eta \in \mathcal{Y}^{}\left( [a, b ]_{p},
	%TCIMACRO{\U{211d} }%
	%BeginExpansion
	\mathbb{R}
	%EndExpansion
	\right) $ is an admissible variation for problem \ref{problem} if $\eta (a)=0=\eta
	(b). $
\end{definition}

\subsection{Basic Lemmas\protect \bigskip}
In order to get our results, we need important lemma:
\begin{lemma}[\textbf{Fundamental Lemma of }$p$\textbf{-variational Calculus}]\label{fundlemmavar}
	Let $f\in \mathcal{Y}^{}([a,b]_{p},%
	%TCIMACRO{\U{211d} }%
	%BeginExpansion
	\mathbb{R}
	%EndExpansion
	)$. One has%
	\begin{equation*}
	\int \limits_{a}^{b}f(t)h(t^{p})d_{p}t=0
	\end{equation*}%
	for all functions $h\in \mathcal{Y}^{}$ with $h(a)=h(b)=0$ if and only if $%
	f(t)=0$ for all $t\in \lbrack a,b]_{p}.$
\end{lemma}

\begin{proof}
	The implication "$\Leftarrow $" is obvious. Let us prove the implication "$%
	\Rightarrow $". Conversely, suppose that exists $q\in \lbrack
	a,b]_{p}$ such that $f(q)\neq 0$.
	
	\begin{enumerate}
		\item If $q\neq 0$, then $q=a^{p^{-j}}$ or $q=b^{p^{-j}}$ for some $j\in
		%TCIMACRO{\U{2115} }%
		%BeginExpansion
		\mathbb{N}
		%EndExpansion
		_{0}.$\newline
		
		\begin{enumerate}
			\item Suppose that $a\neq 0$ and $b\neq 0$. In this case we can assume,
			without loss of generality, that $q=b^{p^{j}}$. Define%
			\begin{equation*}
			h(t)=\left \{
			\begin{array}{cc}
			f(a^{p^{-j-1}}),\text{ } & t=a^{p^{-j}}, \\
			0, & \text{otherwise.}%
			\end{array}%
			\right.
			\end{equation*}%
			Then%
			\begin{align*}
			\int \limits_{a}^{b}f(t)h(t^{p})d_{p}t =&\int
			\limits_{0}^{b}f(t)h(t^{p})d_{p}t-\int \limits_{0}^{a}f(t)h(t^{p})d_{p}t \\
			=&\int
			\limits_{0}^{1}f(t)h(t^{p})d_{p}t+\int
			\limits_{1}^{b} f(t)h(t^{p})d_{p}t-\int \limits_{0}^{a}f(t)h(t^{p})d_{p}t \\
			=&\sum \limits_{j=-\infty
			}^{\infty }(p^{p^{j+1}}-p^{p^{j}})f(p^{p^{j}})h((p^{p^{j}})^p)+\sum \limits_{j=0}^{\infty
			}(b^{p^{-j}}-b^{p^{-j-1}})f(b^{p^{-j-1}})h(\left( b^{p}\right)
			^{p^{-j-1}})\\
			&-\sum \limits_{j=0}^{\infty
			}(a^{p^{-j}}-a^{p^{-j-1}})f(a^{p^{-j-1}})h(\left( a^{p}\right) ^{p^{-j-1}})
			\\
			=&-(a^{p^{-j}}-a^{p^{-j-1}})\left[ f(a^{p^{-j-1}})\right] ^{2}\neq 0,
			\end{align*}%
			which is a contradiction.
			
			\item Suppose that $a=0$ and $b\neq 0$, then $q=b^{p^{-j}}$for some $%
			j\in
			%TCIMACRO{\U{2115} }%
			%BeginExpansion
			\mathbb{N}
			%EndExpansion
			_{0}.$Define%
			\begin{equation*}
			h(t)=\left \{
			\begin{array}{cc}
			f(b^{p^{-j-1}}),\text{ } & t=b^{p^{-j}}, \\
			0, & \text{otherwise.}%
			\end{array}%
			\right.
			\end{equation*}%
			and as in the proof (a), we obtain a contradiction.
			
			\item The case $b=0$ and $a\neq 0$ is similar to the previous one.
		\end{enumerate}
		
		\item If $q=0$, without loss of generality, we can assume that $f(q)>0.$
		Since
		\begin{equation*}
		\lim_{j\rightarrow \infty }a^{p^{-j-1}}=\lim_{j\rightarrow \infty
		}b^{p^{-j-1}}=1
		\end{equation*}%
		and $f$ is continuous at $1$, we have%
		\begin{equation*}
		\lim_{j\rightarrow \infty }f(a^{p^{-j-1}})=\lim_{j\rightarrow \infty
		}f(b^{p^{-j-1}})=f(1).
		\end{equation*}%
		Therefore, there exists an order $j_{0}\in
		%TCIMACRO{\U{2115} }%
		%BeginExpansion
		\mathbb{N}
		%EndExpansion
		$ such that for all $j>j_{0}$ the inequalities%
		\begin{equation*}
		f(a^{p^{-j-1}})>1\text{ and }f(b^{p^{-j-1}})>1
		\end{equation*}%
		hold.
		
		\begin{enumerate}
			\item If $a,b\neq 0$, then for some $k>j_{0}$, we define
			\begin{equation*}
			h(t)=\left \{
			\begin{array}{cc}
			f(b^{p^{-j-1}}),\text{ } & \text{if }t=b^{p^{-j}}, \\
			f(a^{p^{-j-1}}), & \text{if }t=a^{p^{-j}}, \\
			0, & \text{otherwise.}%
			\end{array}%
			\right.
			\end{equation*}%
			Hence%
			\begin{equation*}
			\int%
			\limits_{a}^{b}f(t)h(t^{p})d_{p}t=(b^{p^{-j}}-b^{p^{-j-1}})f(b^{p^{-j-1}})-(a^{p^{-j}}-a^{p^{-j-1}})f(a^{p^{-j-1}})\neq 0.
			\end{equation*}
			
			\item If $a=0$, then we define
			\begin{equation*}
			h(t)=\left \{
			\begin{array}{cc}
			f(b^{p^{-j-1}}),\text{ } & \text{if }t=b^{p^{-j}}, \\
			0, & \text{otherwise.}%
			\end{array}%
			\right.
			\end{equation*}%
			Therefore
			\begin{equation*}
			\int \limits_{0}^{b}f(t)h(t^{p})d_{p}t=(b^{p^{-j}}-b^{p^{-j-1}})\left[
			f(b^{p^{-j-1}})\right] ^{2}\neq 0.
			\end{equation*}
			
			\item If $b=0$, this follows by the same method as in the previous case.
		\end{enumerate}
	\end{enumerate}
\end{proof}

\begin{definition}\label{uniformdefn}
	Let $\ s\in I$ and $g:I\times (-\theta ,\theta )\rightarrow
	%TCIMACRO{\U{211d} }%
	%BeginExpansion
	\mathbb{R}
	%EndExpansion
	.$ We say that $g(t,\cdot )$ is differentiable at $\theta _{0}$ uniformly in $%
	[s]_{p}$ if for every $\epsilon >0$ there exists $\delta >0$ such that
	\begin{equation*}
	0<|\theta -\theta _{0}|<\delta \Longrightarrow \left \vert \frac{g(t,\theta
		)-g(t,\theta _{0})}{\theta -\theta _{0}}-\partial _{2}g(t,\theta _{0})\right
	\vert <\epsilon
	\end{equation*}%
	for all $t\in \lbrack s]_{p},$ where $\partial _{2}g=\frac{\partial g}{%
		\partial \theta }.$
\end{definition}

\begin{lemma}[Cf. \cite{malinowska2010hahn}]\label{uniformlemma}
	Let $\ s\in I$ and \ assume that $g:I\times (-\theta ,\theta )\rightarrow
	%TCIMACRO{\U{211d} }%
	%BeginExpansion
	\mathbb{R}
	%EndExpansion
	$ is differentiable at $\theta _{0}$ uniformly in $[s]_{p}.$ If $\int
	\limits_{0}^{s}g(t,\theta _{0})d_{p}t$ exists, then $G(\theta ):=\int
	\limits_{0}^{s}g(t,\theta )d_{p}t$ \ for $\theta $ near $\theta _{0}$ is
	differentiable at $\theta _{0}$ and%
	\begin{equation*}
	G^{\prime }(\theta _{0})=\int \limits_{0}^{s}\partial _{2}g(t,\theta
	_{0})d_{p}t.
	\end{equation*}
\end{lemma}

\begin{proof}
	
	\begin{enumerate}
		\item[i.)]	Let $ s<1 $ be. Since $g(t,\cdot)$ is differentiable at $\theta _{0}$ uniformly in $[s]_{p}$,
		then for every $\epsilon >0$ there exists $\delta >0$ such that for
		all $t\in $ $[s]_{p}$ and for $0<|\theta -\theta _{0}|<\delta $, the
		following inequalities hold:%
		\begin{equation*}
		\left \vert \frac{g(t,\theta )-g(t,\theta _{0})}{\theta -\theta _{0}}%
		-\partial _{2}g(t,\theta _{0})\right \vert <\frac{\epsilon }{2s},
		\end{equation*}%
		\begin{eqnarray*}
			\left \vert \frac{G(\theta )-G(\theta _{0})}{\theta -\theta _{0}}-G^{\prime
			}(\theta _{0})\right \vert &\leq &\int \limits_{0}^{s}\left \vert \frac{%
				g(t,\theta )-g(t,\theta _{0})}{\theta -\theta _{0}}-\partial _{2}g(t,\theta
			_{0})\right \vert d_{p}t \\
			&<&\int \limits_{0}^{s}\frac{\epsilon }{2s}d_{p}t=\frac{\epsilon }{2}%
			<\epsilon .
		\end{eqnarray*}
		\item[ii.)]	
		Let $ s>1 $ be. Since $g(t,\cdot)$ is differentiable at $\theta _{0}$ uniformly in $[s]_{p}$,
		then for every $\epsilon >0$ there exists $\delta >0$ such that for
		all $t\in $ $[s]_{p}$ and for $0<|\theta -\theta _{0}|<\delta $, the
		following inequalities hold:%
		\begin{equation*}
		\left \vert \frac{g(t,\theta )-g(t,\theta _{0})}{\theta -\theta _{0}}%
		-\partial _{2}g(t,\theta _{0})\right \vert <\frac{\epsilon }{2(s-1)},
		\end{equation*}%
		\begin{eqnarray*}
			\left \vert \frac{G(\theta )-G(\theta _{0})}{\theta -\theta _{0}}-G^{\prime
			}(\theta _{0})\right \vert &\leq &\int \limits_{0}^{s}\left \vert \frac{%
				g(t,\theta )-g(t,\theta _{0})}{\theta -\theta _{0}}-\partial _{2}g(t,\theta
			_{0})\right \vert d_{p}t \\
			&<&\int \limits_{0}^{s}\frac{\epsilon }{2(s-1)}d_{p}t=\frac{\epsilon }{2}%
			<\epsilon .
		\end{eqnarray*}
	\end{enumerate}
	
	Hence, $G(\cdot)$ is differentiable at $\theta _{0}$ and $G^{\prime }(\theta
	_{0})=\int \limits_{0}^{s}\partial _{2}g(t,\theta _{0})d_{p}t.$
\end{proof}

\subsection{$p$-variational Problem}
For an admissible variation $\eta $ and an admissible function $y$, \ we
define the real function $\phi $ by%
\begin{equation*}
\phi (\varepsilon )=\phi (\varepsilon ,y,\eta ):=\mathcal{L} \lbrack
y+\varepsilon \eta ].
\end{equation*}%
The first variation of the functional \bigskip $\mathcal{L} $ of the problem
\ref{problem} is defined by%
\begin{equation*}
\delta \mathcal{L} \lbrack y,\eta ]:=\phi ^{\prime }(0).
\end{equation*}%
Note that
\begin{align*}
\mathcal{L} \lbrack y+\varepsilon \eta ] =&\int
\limits_{a}^{b}L\bigg(t,y(t^{p})+\varepsilon \eta (t^{p}),D_{p}[y](t)+\varepsilon
D_{p}[\eta ](t)\bigg)d_{p}t \\
=&\mathcal{L} _{b}[y+\varepsilon \eta ]-\mathcal{L} _{a}[y+\varepsilon \eta
]
\end{align*}%
where
\begin{equation*}
\mathcal{L} _{\xi }[y+\varepsilon \eta ]=\int \limits_{0}^{\xi
}L\bigg(t,y(t^{p})+\varepsilon \eta (t^{p}),D_{p}[y](t)+\varepsilon D_{p}[\eta
]\bigg)(t)d_{p}t
\end{equation*}%
with $\xi \in \{a,b\}$. Therefore,%
\begin{equation*}
\delta \mathcal{L} \lbrack y,\eta ]=\delta \mathcal{L} _{b}[y,\eta ]-\delta
\mathcal{L} _{a}[y,\eta ].
\end{equation*}

The following lemma is direct consequence of Lemma \ref{uniformlemma}. In what follows $%
\partial _{i}L$ denotes the partial derivative of $L$ with respect to its $i$%
th argument.

\begin{lemma}\label{lemadm}
	For an admissible variation $\eta $ and an admissible function $y.$ Let
	\begin{equation*}
	g(t,\varepsilon ):=L\bigg(t,y(t^{p})+\varepsilon \eta
	(t^{p}),D_{p}[y](t)+\varepsilon D_{p}[\eta ](t)\bigg).
	\end{equation*}%
	Assume that
	
	\begin{enumerate}
		\item $g(t,\cdot)$ is differentiable at $0$ uniformly in $[a , b ]_{p};$
		
		\item $\mathcal{L} _{a}[y+\varepsilon \eta ]=\int
		\limits_{0}^{a}g(t,\varepsilon )d_{p}t$ and $\mathcal{L} _{b}[y+\varepsilon
		\eta ]=\int \limits_{0}^{b}g(t,\varepsilon )d_{p}t$ exist for $\varepsilon
		\approx 0;$
		
		\item $\int \limits_{0}^{a}\partial _{2}g(t,0)d_{p}t$ and $\int
		\limits_{0}^{b}\partial _{2}g(t,0)d_{p}t$ exists.
	\end{enumerate}
	
	Then
	\begin{align*}
	\phi ^{\prime }(0)=&\delta \mathcal{L} \lbrack y,\eta ]\\
	=&\int
	\limits_{a}^{b}\Big( \partial _{2}L\big(t,y(t^{p}),D_{p}[y](t)\big)\eta
	(t^{q})+\partial _{3}L\big(t,y(t^{p}),D_{p}[y](t)\big)D_{p}\eta (t^{{}})\Big)
	d_{p}t.
	\end{align*}
\end{lemma}

\subsection{Optimality conditions}

In this section, we will present a necessary condition (the $p$-Euler-Lagrange equation) and a sufficient condition to our Problem \ref{problem}.

\begin{theorem}[The $p$-Euler-Lagrange equation] \label{theopeulerlagrange}
	Under hypotheses (H1)-(H3) and conditions 1-3 of Lemma \ref{lemadm} on the Lagrangian $%
	L $, if $y_{\ast }\in \mathcal{Y}^{}$ is a local extremizer for problem
	\ref{problem}, then $y_{\ast }$ satisfies the $p$-Euler-Lagrange equation%
	\begin{equation}\label{peulerlagrange}
	\partial _{2}L\big(t,y(t^{p}),D_{p}[y](t)\big)=D_{p}\big[ \partial _{3}L(\cdot
	,y(\cdot ^{p}),D_{p}[y](\cdot ))\big] (t)
	\end{equation}%
	for all $t\in \lbrack a, b]_{p}.$
\end{theorem}

\begin{proof}
	Let $y_{\ast }$ be a local minimizer (resp. maximizer) for problem \ref{problem} and $%
	\eta $ an admissible variation. Define $\phi :%
	%TCIMACRO{\U{211d} }%
	%BeginExpansion
	\mathbb{R}
	%EndExpansion
	\rightarrow
	%TCIMACRO{\U{211d} }%
	%BeginExpansion
	\mathbb{R}
	%EndExpansion
	$ by
	\begin{equation*}
	\phi (\varepsilon ):=\mathcal{L} \lbrack y_{\ast }+\varepsilon \eta ].
	\end{equation*}%
	A necessary condition for $y_{\ast }$ to be an extremizer is given by $\phi
	^{\prime }(0)=0.$ By Lemma \ref{lemadm}, we conclude that
	\begin{equation*}
	\int \limits_{a}^{b}\Big( \partial _{2}L\big(t,y(t^{p}),D_{p}[y](t)\big)\eta
	(t^{q})+\partial _{3}L\big(t,y(t^{p}),D_{p}[y](t)\big)D_{p}\eta (t)\Big)
	d_{p}t=0
	\end{equation*}%
	By integration by parts \eqref{byparts}, we get%
	\begin{eqnarray*}
		&&\int \limits_{a}^{b}\partial _{3}L\big(t,y(t^{p}),D_{p}[y](t)\big)D_{p}\eta
		(t^{{}})d_{p}t \\
		&=&\partial _{3}L\big(t,y(t^{p}),D_{p}[y](t)\big)\eta (t^{{}})|_{a}^{b}-\int
		\limits_{a}^{b}D_{p}\partial _{3}L\big(\cdot ,y(\cdot ^{p}),D_{p}[y](\cdot
		)\big)(t)\eta (t^{p})d_{p}t.
	\end{eqnarray*}%
	Since $\eta (a)=\eta (b)=0$, then%
	\begin{equation*}
	\int \limits_{a}^{b}\bigg( \partial _{2}L\big(t,y(t^{p}),D_{p}[y](t)\big)\eta
	(t^{p})-D_{p}\partial _{3}L\big(\cdot ,y(\cdot ^{p}),D_{p}[y](\cdot )\big)(t)\eta
	(t^{p})\bigg) d_{p}t=0
	\end{equation*}%
	Finally, by Lemma \ref{fundlemmavar}, for all $t\in \lbrack a, b ]_{p}$%
	\begin{equation*}
	\partial _{2}L\bigg(t,y(t^{p}),D_{p}[y](t)\bigg)=D_{p}\partial _{3}L\bigg(\cdot ,y(\cdot
	^{p}),D_{p}[y](\cdot )\bigg)(t).
	\end{equation*}
\end{proof}

To conclude this section, we prove a sufficient optimality condition for the
Problem \ref{problem}.

\begin{definition}
	Given a function $L$,we say that $L(t,u,v)$ is jointly convex(resp. concave)
	in $(u,v)$ iff $\partial _{i}L$, $i=2,3$, exist and continuous and verify the
	following conditions:%
	\begin{equation*}
	L(t,u+u_{1},v+v_{1})-L(t,u,v) ~\geq{(\text {resp.~}\leq )}~\partial
	_{2}L(t,u,v)u_{1}+\partial _{3}L(t,u,v)v_{1}
	\end{equation*}%
	for all $(t,u,v),(t,u+u_{1},v+v_{1})\in I\times
	%TCIMACRO{\U{211d} }%
	%BeginExpansion
	\mathbb{R}
	%EndExpansion
	^{2}.$
\end{definition}

\begin{theorem}\label{convextheo}
	Suppose that $a<b$ and $a,b\in \lbrack c]_{p}$ for some $c\in I$. Also
	assume that $L$ is a jointly convex(resp. concave) function in $(u,v).$ If $%
	y_{\ast }$ satisfies the $p$-Euler-Lagrange equation \eqref{peulerlagrange}, then $y_{\ast }$ is
	global minimizer (resp. maximizer) to the problem \ref{problem}.
\end{theorem}

\begin{proof}
	Let $L$ be a jointly convex function in $(u,v)$. Then for any admissible variation $\eta $, we have
	\begin{align*}
	\mathcal{L} \lbrack y_{\ast }+\eta ]-\mathcal{L} \lbrack y_{\ast }] &=\int
	\limits_{a}^{b}\Bigg[ L\bigg(t,y_{\ast }(t^{p})+\eta (t^{p}),D_{p}[y_{\ast
	}](t)+D_{p}[\eta ](t)\bigg)-L\bigg(t,y_{\ast }(t^{p}),D_{p}[y_{\ast }](t)\bigg)\Bigg]
	d_{p}t \\
	&\geq \int \limits_{a}^{b}\Bigg[ \partial _{2}L\bigg(t,y_{\ast
	}(t^{p}),D_{p}[y_{\ast }](t)\bigg)\eta (t^{p})+\partial _{3}L\bigg(t,y_{\ast
	}(t^{p}),D_{p}[y_{\ast }](t)\bigg)D_{p}[\eta ](t)\Bigg] d_{p}t.
	\end{align*}%
	Using integrations by part, formula \eqref{byparts}, we get%
	\begin{align*}
	\mathcal{L} \lbrack y_{\ast }+\eta ]-\mathcal{L} \lbrack y_{\ast }] &\geq
	\int \limits_{a}^{b}\Big[ \partial _{2}L\big(t,y_{\ast }(t^{p}),D_{p}[y_{\ast
	}](t)\big)\eta (t^{p})\Big] d_{p}t \\
	&\quad +\int \limits_{a}^{b}\Big[ \partial
	_{3}L\big(t,y_{\ast }(t^{p}),D_{p}[y_{\ast }](t)\big)D_{p}[\eta ](t)\Big] d_{p}t \\
	&\geq \int \limits_{a}^{b}\Big[ \partial _{2}L\big(t,y_{\ast
	}(t^{p}),D_{p}[y_{\ast }](t)\big)\eta (t^{p})\Big] d_{p}t \\
	&\quad +\partial _{3}L\big(t,y_{\ast }(t^{p}),D_{p}[y_{\ast }](t)\big)\eta
	(t)|_{a}^{b} \\
	&\quad -\int \limits_{a}^{b}D_{p}\partial _{3}L\big(\cdot ,y(\cdot
	^{p}),D_{p}[y](\cdot )\big)(t)\eta (t^{p})d_{p}t.
	\end{align*}%
	Since $y_{\ast }$ satisfies Theorem \ref{theopeulerlagrange} and $\eta $ is an admissible variation, we
	obtain%
	\begin{equation*}
	\mathcal{L} \lbrack y_{\ast }+\eta ]-\mathcal{L} \lbrack y_{\ast }]\geq 0,
	\end{equation*}%
	proving that $y_{\ast }$ is a minimizer of Problem \ref{problem}.
	The same conclusion can be drawn for the concave case.
\end{proof}

\section{An example}
We require that $ p $ is a fixed number different from $ 1 $. For $ a<b $ in $ [a,b]_{p} $, consider the following problem
\begin{equation}
\begin{cases}
\mathcal{L} \lbrack y ] = \int \limits_{a}^{b} \left(t+\frac{1}{2}\left( D_p [y](t)\right)^2\right)d_p t \rightarrow \text{~minimize}\\
y \in \mathcal{Y}^{}([a,b]_{p}, \mathbb{R}) \\
y(a)=a, \\
y(b)=b.
\end{cases}
\end{equation}
If $ y_{\ast } $ is a local minimizer of the problem, then $ y_{\ast } $  satisfies the $ p $-Euler-Lagrange equation:
\begin{equation}
D_p\left[ D_p[y](p\cdot)\right](t)=0 \text{~ for all } t \in [a,b]_{p}.
\end{equation}
It can be easily seen that the function $ y_{\ast }(t)=t $ is a candidate to the solution of this problem. Since the Lagrangian function is jointly convex in $ (u,v) $, then by using Theorem \ref{convextheo}, it follows immediately that the function $ y_{\ast} $ is a minimizer of the problem.

\bibliographystyle{plain}

\end{document}